\documentclass[a4,11pt, leqno]{article}
\usepackage{amsmath,amsfonts,amsthm,latexsym,amssymb,mathrsfs}
\def\H_0{\mathcal{H}_0(T)}

\def\X{{\mathcal X}}

\def\ind{{\textrm{ind}}}
\def\asc{{\textrm{asc}}}

\def\iso{{\textrm{iso }}}
\def\ST{\tau_{AB}}

\def\ind{\textrm{ind}}
\def\X{{\cal X}}
\def\Y{{\cal Y}}
\def\H{{\cal H}}

\def\iso{\textrm{iso }}

\def\asc{ \textrm{asc}}

\newtheorem{df}{Definition}[section]
\newtheorem{thm}[df]{Theorem}
\newtheorem{pro}[df]{Proposition}

\newtheorem{rema}[df] {Remark}
\newtheorem{lem}[df] {Lemma}
\def\sfstp{{\hskip-1em}{\bf.}{\hskip1em}}

\def\subject#1{\renewcommand{\thefootnote}{}
\footnote{\it AMS Subject Classification \rm (2010): {#1}}}

\def\keywords#1{\renewcommand{\thefootnote}{}
\footnote{ \it Key words and phrases: {#1}}}

\def\enddemo{\qed \endtrivlist}
\expandafter\let\csname enddemo*\endcsname=\enddemo

\def\qedsymbol{\ifmmode\bgroup\else$\bgroup\aftergroup$\fi
  \vcenter{\hrule\hbox{\vrule
height.6em\kern.6em\vrule}\hrule}\egroup}
\def\qed{\ifmmode\else\unskip\nobreak\fi\quad\qedsymbol}

\begin{document}
\title
{ \bf Property $(gw)$ for tensor product and\\ left-right multiplication operators\/}

\author {\normalsize ENRICO  BOASSO and B. P. DUGGAL } \vskip 1truecm

\date{ }


\maketitle \thispagestyle{empty} 

\subject{Primary 47A80, 47A53. Secondary 47A10.} \keywords{  \rm Banach space,
property $(gw)$, tensor product operator, left-right
multiplication\newline
\indent $\hbox{                      \hskip.19truecm }$operator.}

\vskip 1truecm

\setlength{\baselineskip}{12pt}

\begin{abstract} Given Banach spaces $\X$ and $\Y$ and operators $A\in B(\X)$
and $B\in B(\Y)$, property $(gw)$ does not in general transfer from $A$ and $B$
to the tensor product operator $A\otimes B\in B(\X\overline{\otimes} \Y)$ or to the
elementary operator defined by $A$ and $B$, $\ST=L_AR_B\in B(B(Y,\X))$. In this article necessary
and sufficient conditions ensuring that property $(gw)$ transfers from
$A$ and $B$ to $A\otimes B$ and to $\ST$ will be given.
\end{abstract}


\section {\sfstp Introduction}\setcounter{df}{0}
\ \indent A Banach space operator satisfies property $(gw)$ if
the complement of its upper $B$-Weyl spectrum
in its approximate point spectrum is its set of isolated eigenvalues. 
Originally introduced by V. Rako\v
cevi\'c in \cite{R}, property $(w)$ has been intensively studied
recently, see the articles by P. Aiena et al. [1-5]. On the other
hand, M. Amouch and M. Berkani have introduced and studied
property $(gw)$, an extension of property $(w)$, see \cite{ABe,
BA, Am}. Note that property $(w)$ and $(gw)$ are variants of
Weyl's and generalized Weyl's theorems and they are also related
to $a$-Browder's and generalized $a$-Browder's theorems.\par
\indent In \cite{D}, given $\X$ and $\Y$ two Banach spaces and
$A\in B(\X)$ and $B\in B(\Y)$ two Banach space operators, the
second author gave necessary and sufficient conditions to
characterize when $A\otimes B$ satisfies property $(w)$, where
$A\otimes B\in B(\X\overline{\otimes} \Y)$ is the tensor product
operator defined on $\X\overline{\otimes} \Y$, the completion of
the algebraic tensor product $\X\otimes \Y$ of $\X$ and $\Y$ with
respect to a reasonable cross norm. In particular, when $A$ and
$B$ are isoloid operators that satisfy property $(w)$, necessary
and sufficient for $A\otimes B$ to satisfy property $(w)$ is that
the   $a$-Weyl spectrum equality for \rm $A\otimes B$ holds, see
\cite[Theorem 1]{D} and \cite{DDK}.\par \indent The main
objective of this work is to give necessary and sufficient
conditions to characterize when property $(gw)$ holds for
$A\otimes B\in B(\X\overline{\otimes} \Y)$, where the Banach
space operators $A\in B(\X)$ and $B\in B(\Y)$ satisfy property
$(gw)$. In particular,  property $(gw)$ for $A\otimes B$ will be
related to the generalized $a$-Weyl spectrum inclusion, see
section 3. A similar argument  leads to a characterization of
property $(gw)$ for the left-right multiplication operator
$\ST=L_AR_B\in B(B(\Y,\X))$. Finally, in section 2, after having
recalled the key notions of this article, the main known
characterizations of properties $(w)$ and
 $(gw)$ will be proved using arguments which differ from the original ones.
\par

\section {\sfstp Preliminary definitions and results}\setcounter{df}{0}
 \hskip.5truecm
\indent
\markright{ \hskip5truecm \rm ENRICO BOASSO and B. P. DUGGAL}
\indent From now on $\X$ shall denote an
infinite dimensional complex Banach space and $B(\X)$ the algebra of
all bounded linear maps defined on and with values in $\X$. Given $A\in
B(\X)$, $N(A)$ and $R(A)$ will stand for the null space and the
range of $A$ respectively. Recall that $A\in B(\X)$ is said to be
\it bounded below\rm, if $N(A)=0$ and $R(A)$ is closed. Denote
the \it approximate point spectrum \rm of $A$ by
$\sigma_a(A)=\{\lambda\in \mathbb C \colon A-\lambda \hbox{ is
not bounded below} \}$, where $A-\lambda$ stands for $A-\lambda
I$, $I$ the identity map of $B(\X)$. Let
$\sigma_s(A)=\{\lambda\in \mathbb C \colon R(A-\lambda)\neq \X\}$
denote
 the \it surjectivity spectrum \rm of $A$. Clearly,  $\sigma_a(A)\cup  \sigma_s(A)
=\sigma(A)$, the spectrum of $A$. \par

\indent Recall that $A\in B(\X)$  is said to be  a
\it Fredholm \rm operator if $\alpha(A)=\dim N(A)$ and $\beta(A)=\dim \X/R(A)$
are finite dimensional, in which case its \it index \rm is given by
$$
\ind(A)=\alpha(A)-\beta (A).
$$
If $R(A)$ is closed and $\alpha (A)$ is finite,
then $A\in  B(\X)$ is said to be  \it upper  semi-Fredholm \rm
(note that in this case ind $(A)$ is well defined),
while if $\alpha (A)$ and $\beta (A)$ are finite and equal, so that the index is zero,
$A$ is said to be \it Weyl. \rm
These classes of operators generate the
Fredholm or essential spectrum, the upper semi-Fredholm spectrum and the Weyl spectrum of $A\in B(\X)$, which will be denoted by
 $\sigma_e(A)$, $\sigma_{SF_+}(A)$ and $\sigma_w(A)$,
 respectively.   In addition, the \it Weyl essential approximate point spectrum
\rm of $A\in B(\X)$ is the set $\sigma_{aw}(A)=\{\lambda\in \sigma_a(A)\colon
\lambda\in \sigma_{SF_+}(A)\hbox{ or } 0<\ind (A-\lambda) \}$ (\cite{R1,R2}). \par


\indent Consider $A\in B(\X)$ and define $\Delta (A)=\sigma(T)\setminus\sigma_w(T)$.
Recall that according to \cite{C}, \it Weyl's theorem \rm holds for $A\in B(\X)$, if $\Delta (A)=E_0(A)$,
where $E_0(A)= \{\lambda\in\hbox{\rm iso } \sigma(A)\colon
0<\alpha (A-\lambda)<\infty\}$. Here and elsewhere in this article,
for $K\subseteq  \mathbb C$, iso $K$ will stand for the set of isolated points of
$K$ and acc $K=K\setminus$ iso $K$.
\par

\indent In recent years there have been generalizations of the
Fredholm concept. An operator $A\in B(\X)$ will be said to be \it
B-Fredholm\rm, if there exists $n\in\mathbb
N$ for which the range of $R(A^n)$ is closed and the induced
operator $A_n\in B(R(A^n))$ is Fredholm (\cite{B1}). In a similar
way it is possible to define upper  B-Fredholm operators
(\cite{B2}). Note that if for some $n\in\mathbb N$, $A_n\in
B(R(A^n))$ is Fredholm, then $A_m\in B(R(A^m))$ is Fredholm for
all $m\ge n$; moreover $\ind (A_n)=\ind (A_m)$, for all $m\ge n$.
Therefore, it makes sense to define the index of $A$ by $\ind
(A)=\ind (A_n)$. Recall that $A$ is said to be \it B-Weyl \rm  if
$A$ is B-Fredholm and $\ind(A)=0$. Naturally, from this class of
operators
the B-Weyl spectrum of $A\in B(\X)$ can be derived,
which will be denoted by
$\sigma_ {BW}(A)$ (\cite{B3}). In addition, the upper $B$-Weyl spectrum of $A$ is the set $\sigma_ {SBF_+^-}(A)= \{\lambda\in
\mathbb C\colon A-\lambda  \hbox{ is not upper semi B-Fredholm or
} 0<\ind (A-\lambda)\}$  (\cite{BK}).\par

\indent Let  $A\in B(\X)$ and consider the sets $\Delta^g(A) =\sigma (A)\setminus \sigma_ {BW}(A)$
and $E(A)= \{\lambda\in
\hbox{\rm iso } \sigma(A)\colon 0<\alpha(A-\lambda)\}$.
Recall that \it generalized Weyl's theorem \rm holds for $A\in B(\X)$,
if $\Delta^g(A)= E(A)$ (\cite{BK}).\par

\indent Consider $A\in B(\X)$ and define $\Delta_a(A)= \sigma_a
(A)\setminus \sigma_{aw}(A)$ and
$\Delta^g_a(A)=\sigma_a(A)\setminus \sigma_{SBF_+^-}(A)$. The
operator $A$  is said to satisfy \it property (w)\rm, if
$\Delta_a(A)= E_0(A)$. For further information on key properties
of this notion, see \cite{R} and  [1-5]. Next follows the main
definition of this work, see \cite{ ABe, BA, Am}.

\begin{df} Let $\X$ be a Banach space and consider $A\in B(\X)$. The operator
$A$ is said to satisfy propety $(gw)$ if
$$
\Delta^g_a(A)=E(A).
$$
\end{df}

\indent To recall some of the most relevant results related to
propeties $(w)$ and $(gw)$, first of all several notions need to be
recalled.\par

\indent The \it ascent \rm  (respectively \it the descent\rm ) of
$A\in B(\X)$ is the smallest non-negative integer $a=asc (A)$
(respectively $d=dsc (A)$) such that $N(A^a)=N(A^{a+1})$ (respectively
$R(A^d)=R(A^{d+1})$); if such an integer does not exist, then
$asc(A)=\infty$ (respectively $dsc(A)=\infty$). Recall
that $\lambda\in \sigma (A)$ is said to be \it a pole \rm of
$A$, if the ascent and the descent of $A-\lambda$ is finite.
The set of all poles of $A\in B(\X)$ will be denoted by
$\Pi (A)$. In addition, the set of \it poles of finite rank \rm of $A$ is the set
$\Pi_0(A)=\{\lambda\in \Pi(A)\colon \alpha(A-\lambda)<\infty\}$.
\par

\indent Recall that a Banach space operator $A\in B(\X)$ is said to be \it Drazin invertible\rm,
if there exists a necessarily unique $B\in B(\X)$ and some $m\in \mathbb N$ such that
$$
A^m=A^mBA, \hskip.3truecm BAB=B, \hskip.3truecm AB=BA.
$$
According to \cite[Theorem 4]{K}, necessary and sufficient for $A\in B(\X)$ to be
Drazin invertible is that $0\in \Pi (A)$.  If $DR(B(\X))=\{ A\in B(\X)\colon
A\hbox{ is Drazin invertible} \}$, then the \it Drazin spectrum \rm of
$A\in B(\X)$ is the set $\sigma_{DR}(A)=\{\lambda\in \mathbb C\colon
A-\lambda\notin DR(B(\X) \}$ (\cite{BS, Bo}). \par


\indent Next denote by $LD(\X)= \{ A\in B(\X)\colon \hbox{  }a=
\asc(A)<\infty\hbox{  and } R(A^{a+1 }) \hbox{ is closed}\}$ the
set of \it left Drazin invertible \rm operators. Then, given
$A\in B(\X)$, the \it left Drazin spectrum \rm of $A$ is the set
$\sigma_{LD}(A)= \{\lambda\in \mathbb C\colon A-\lambda \notin
LD(\X)\}$. Note that according to \cite[Lemma 2.12]{BK}, $\sigma_
{SBF_+^-}(A)\subseteq \sigma_{LD}(A)\subseteq \sigma_a(A)$.\par

\indent In addition, $\lambda\in\sigma_a(A)$ is said to be a \it
left pole \rm of $A$, if $A-\lambda$ is left Drazin invertible.
The set of all left poles of $A$ will be denoted by $\Pi^l(A)$
(\cite{BK, AS}). Note that $\Pi (A)\subseteq \Pi^l(A)$
(\cite[Theorem 4 and Corollary]{K}). In addition, the set of \it
left poles of finite rank \rm of $A$ is the set $\Pi_0^l(A)=
\{\lambda\in\Pi ^l(A)\colon \alpha(A-\lambda)<\infty\}$.
 \par

\indent Recall that $A\in B(\X)$ is said  to satisfy \it $a$-Browder's theorem\rm,
 if $\sigma_{aw}(A)=\sigma_a(A)\setminus \Pi_0^l(A)$, while $A$ is said to satisfy
\it generalized $a$-Browder's theorem\rm, if  $\sigma_{SBF_+^-}(A)=\sigma_a(A)\setminus \Pi^l(A)$.\par

\indent Next some characterizations of operators for which properties $(w)$ or
$(gw)$ hold will be presented. However, in first place some preparation is needed.\par

\begin{lem}\label{lem1} Let $\X$ be a Banach space and consider $A\in B(\X)$.
Then, $\Pi^l(A)\setminus \Pi(A)\subseteq \hbox{ \rm acc }\sigma (A)$.\end{lem}
\begin{proof} If $\lambda\in \Pi^l(A)\setminus \Pi(A)$, then
$\lambda\in \iso\sigma_a(A)$ and
$\lambda\notin\iso\sigma(A)$, which implies that $\lambda\in\rm{acc}\sigma(A)$.\end{proof}

\indent In the following proposition some characterizations of
operators satisfying property $(w)$ will be given. The results
are known, but the proofs differs from the original ones, compare
with \cite{R, AP}. Note that in what follows, if $A\in B(\X)$,
$\X$ a Banach space, then $\partial\sigma(A)$ will stand for the
topological boundary of $\sigma(A)$.\par
\markright{ \hskip5truecm \rm Property $(gw)$}
\begin{pro}\label{pro2} Let $\X$ be a Banach space and consider $A\in
B(\X)$. Then, the following statements are equivalent.\par
\noindent \rm (i) \it Property $(w)$ holds for $A$;\par
\noindent \rm (ii) \it $a$-Browder's theorem holds for $A$ and $E_0(A)=\Pi^l_0(A)$;\par
\noindent \rm (iii) \it Weyl's theorem holds for $A$ and $\Delta_a(A)=\Delta (A)$;\par
\noindent \rm (iv) \it Weyl's theorem holds for $A$ and $\Delta_a(A)\subseteq$ \rm iso \it$\sigma(A)$;\par
\noindent \rm (v) \it Weyl's theorem holds for $A$ and $\Delta_a(A)\subseteq\partial\sigma (A)$.
\end{pro}
\begin{proof} (i) $\Rightarrow$ (ii). Suppose that $A$ satisfies property $(w)$,
equivalently $\Delta_a (A)=E_0(A)$. Note that according to \cite[Corollaries 2.2 and 2.3]{R3},
$\Delta_a(A)= \Pi_0^l(A)\cup $ (acc $\sigma_a(A)\setminus \sigma_{aw}(A))$. Since $E_0(A)\subseteq$ iso $\sigma (A)$,
acc $\sigma_a(A)\subseteq \sigma_{aw}(A)$ and $\Pi_0^l(A)=E_0(A)$.
However, according again to \cite[Corollaries 2.2 and 2.3]{R3}, $a$-Browder theorem is equivalent to
acc $\sigma_a(A)\subseteq \sigma_{aw}(A)$.\par

\indent (ii) $\Rightarrow$ (i). If $a$-Browder's theorem holds for $A$, then
$\Delta_a(A)= \Pi_0^l(A)$. Clearly, if in addition $\Pi^l_0(A)=E_0(A)$, then $A$ satisfies poperty $(w)$.\par

\indent (i) $\Rightarrow$ (iii). Note that according to \cite[Theorems 3.2.10 and 4.2.1]{CPY},
if $\lambda\in \Delta_a(A)\setminus \Delta(A)$, then $\lambda$ is an interior point of $\sigma (A)$.
As a result, if property $(w)$ holds,  then $E_0(A)=\Delta(A)=\Delta_a(A)$. \par
\indent (iii) $\Rightarrow$ (iv). If $\Delta (A)=E_0(A)$ and $\Delta_a(A)=\Delta (A)$,
then $\Delta_a(A)\subseteq $ iso $\sigma (A)$.\par

\indent (iv) $\Rightarrow$ (v). Clear.\par

\indent (v) $\Rightarrow$ (i). As before, since  $\lambda\in \Delta_a(A)\setminus \Delta(A)$ is an interior point of $\sigma (A)$,
$\Delta_a(A)=\Delta (A)$. If in addition $\Delta (A)=E_0(A)$, then $\Delta_a(A)=E_0(A)$.
\end{proof}

\indent Next operators satisfying property $(gw)$ will be
characterized. As in the case of property $(w)$, all but two of
the results are known, however, since the proofs differ from the
original ones, they will be given. Compare with  \cite{ABe}. \par

\begin{pro}\label{pro3} Let $\X$ be a Banach space and consider $A\in
B(\X)$. Then, the following statements are equivalent.\par
\noindent \rm (i) \it Property $(gw)$ holds for $A$;\par
\noindent \rm (ii) \it generalized  $a$-Browder's theorem holds for $A$ and $E(A)=\Pi^l(A)$;\par
\noindent \rm (iii) \it property $(w)$ holds for $A$ and $E(A)=\Pi^l(A)$;\par
\noindent \rm (iv) \it generalized Weyl's theorem holds for $A$ and $\Delta_a^g(A)=\Delta^g (A)$;\par
\noindent \rm (v) \it generalized Weyl's theorem holds for $A$ and  $\Delta_a^g(A)\subseteq$ \rm iso \it $\sigma(A)$;\par
\noindent \rm (vi) \it generalized Weyl's theorem holds for $A$ and  $\Delta_a^g(A)\subseteq$ $\partial\sigma (A)$.
\end{pro}

\markright{ \hskip5truecm \rm ENRICO BOASSO and B. P. DUGGAL}

\begin{proof} (i) $\Rightarrow$ (ii). Suppose that property $(gw)$ holds for $A$, equivalently
$\Delta_a^g (A)=E(A)$. Note that according to \cite[Theorem
2.8]{BK}, $\Delta_a^g(A)= \Pi^l(A)\cup$  (acc
$\sigma_a(A)\setminus \sigma_{SBF_+^-}(A))$. Since
$E(A)\subseteq$ iso $\sigma (A)$, acc $\sigma_a(A)\subseteq
\sigma_{SBF_+^-}(A)$ and  $\Pi^l(A)=E(A)$. However, according
again to  \cite[Theorem 2.8]{BK}, generalized $a$-Browder's
theorem is equivalent to acc $\sigma_a(A)\subseteq
\sigma_{SBF_+^-}(A)$.\par (ii) $\Rightarrow$ (i). If $A$
satisfies generalized $a$-Browder's theorem, then $\Delta_a^g(A)=
\Pi^l(A)$. Clearly, if in addition $\Pi^l(A)=E(A)$, then $A$
satisfies property $(gw)$.\par
 (ii) $\Rightarrow$ (iii). It is enough to prove that property $(w)$ holds.
 Since $E(A)=\Pi^l(A)$, $E_0(A)=\Pi^l_0(A)$. Moreover, according to
\cite[Theorem 2.2]{AZ}, $a$-Browder's theorem holds for $A$. Consequently,
according to Proposition \ref{pro2}(ii), property $(w)$ holds.\par
 (iii) $\Rightarrow$ (ii). Since property $(w)$ holds for $A$, according to Proposition \ref{pro2}(ii)
and \cite[Theorem 2.2]{AZ}, $A$ satisfies generalized
$a$-Browder's theorem. \par (i) $\Rightarrow$ (iv). Suppose that
$E(A)=\Delta^g_a (A)$. According to the equivalence between
statements (i) and (ii), $\Pi^l(A)=E(A)\subseteq$ iso
$\sigma(A)$. Therefore, according to Lemma \ref{lem1},
$E(A)=\Pi^l(A)=\Pi(A)$. In addition, according again to the
equivalence between statements (i) and (ii) and \cite[Theorem
3.8]{BK}, generalized Browder's theorem holds for $A$. In
particular, according to \cite[Corollary 2.6]{B3}, generalized
Weyl's theorem holds for $A$. Now well, since property $(gw)$ and
generalized Weyl's theorem hold for $A$,
$\Delta_a^g(A)=E(A)=\Delta^g (A)$.
\par
(iv) $\Rightarrow$ (i). If $E(A)=\Delta^g(A)$ and $\Delta^g(A)=\Delta^g_a (A)$,
then $E(A)=\Delta^g_a (A)$, equivalently property $(gw)$ holds for $A$.\par
(iv) $\Rightarrow$ (v). According to the equivalence between statements (i) and (iv),
$\Delta_a^g(A)=E(A)\subseteq$ iso $\sigma (A)$,\par
\par
(v) $\Rightarrow$ (vi). Clear.\par

(vi) $\Rightarrow$ (i). It is not difficult to prove,  using in particular \cite[Theorem 3.1 and Corollary 3.2]{BS2},
that if $\lambda\in \Delta_a^g(A)\setminus\Delta^g (A)$, then $\lambda$ is an interior point of $\sigma (A)$.
Consequently, statement (vi) implies that $\Delta_a^g(A)=\Delta^g (A)$. As a result, according to statement (iv),
$A$ satisfies property $(gw)$.
\end{proof}

\begin{rema}\label{rem4} \rm Let $\X$ be a Banach space and consider $A\in
B(\X)$. Note that according to the proof of Proposition
\ref{pro3}, if $A$ satisfies property $(gw)$, then $E(A)=\Pi
(A)=\Pi^l(A)$.\par \indent Similarly, if property $(w)$ holds for
$A$,   $E_0(A)=\Pi_0 (A)=\Pi^l_0(A)$. In fact, according to
Proposition \ref{pro2}(ii), $\Pi_0 (A)\subseteq \Pi^l_0(A)=
E_0(A)$. However, since $E_0(A)\subseteq$ iso
$\sigma(A)$, according to Lemma \ref{lem1}, $\Pi_0(A)=\Pi_0^l(A)$.
\end{rema}

\indent Recall that a Banach space operator $A\in B(\X)$ is said to be \it
isoloid\rm, if points $\lambda\in $ iso $\sigma (A)$ are eigenvalues of the operator,
equivalently $E(A)=$ iso $\sigma(A)$. In addition,
$A$ is said to be  \it polaroid\rm,
if iso $\sigma (A)=\Pi (A)$. Clearly, a polaroid operator is isoloid.
When property $(gw)$ holds, both notions are equivalent. \par

\begin{lem}\label{lem16} Let $X$ be a Banach space and consider $A\in B(\X)$ such that
property $(gw)$ holds for $A$. Then the following statements are equivalent.\par
\noindent \rm (i)\it $A$ is isoloid;\par
\noindent \rm (ii)\it $A$ is polaroid.
\end{lem}
\begin{proof}Suppose that $A$ isoloid. Then, according to Remark \ref{rem4}, iso $\sigma(A)=E(A)=\Pi (A)$.
\end{proof}

\markright{ \hskip5truecm \rm Property $(gw)$}
\section {\sfstp The transfer property}\setcounter{df}{0}
\
\indent Given Banach spaces $\X$ and $\Y$, $\X\overline{\otimes}\Y$ will stand for
the completion, endowed with a reasonable uniform cross-norm, of the
algebraic tensor product $\X\otimes \Y$ of $\X$ and $\Y$. In addition,
if $A\in B(\X)$ and $B\in B(\Y)$, then $A\otimes B\in B(\X\overline{\otimes}\Y)$ will denote
the tensor product operator defined by $A$ and $B$.\par

\indent In this section the  \it transfer property  for operators
satisfying property $(gw)$ \rm will be studied, i.e., conditions
equivalent to the fact that $A\otimes B\in  B(\X
\overline{\otimes} \Y)$ satisfies property $(gw)$ will be given,
where $A\in B(\X)$ and $B\in B(Y)$ satisfy property $(gw)$. In
particular, the transfer property for operators satisfying
property $(gw)$ will be related  to the \it generalized $a$-Weyl
spectrum inclusion for \rm $A\otimes B$, i.e.,
$$
 \sigma_a(A)\sigma_{SBF_+^-}(B)\cup\sigma_{SBF_+^-}(A)\sigma_a(B)\subseteq\sigma_{SBF_+^-}(A\otimes B).
$$

\indent On the other hand, $\ST\in B(B(\Y, \X))$ will denote the
multiplication operator defined by $A\in B(\X)$ and $B\in B(\Y)$,
i.e., $\ST (U)=AUB$, where $U\in B(\Y, \X)$ and $\X$ and $\Y$ are
two Banach spaces. Note that $\ST=L_AR_B$, where $L_A\in  B(B(\Y,
\X))$ and $R_B\in  B(B(\Y, \X))$ are the left and right
multiplication operators defined by $A$ and $B$ respectively,
i.e., $L_A(U)=AU$ and $R_B(U)=UB$, $U\in  B(\Y, \X)$.\par

\indent The second objective of this section is to study under
what conditions the operator $\ST\in B(B(\Y, \X))$ satisfies
property $(gw)$, assuming that $A\in B(\X)$ and $B\in B(\Y)$
satisfy property $(gw)$. In particular, property $(gw)$ for
$\ST\in B(B(\Y, \X))$ will be related to the \it generalized
$a$-Weyl spectrum inclusion for \rm $\ST$, i.e.,
$$
 \sigma_a(A)\sigma_{SBF_+^-}(B^*)\cup\sigma_{SBF_+^-}(A)\sigma_a(B^*)\subseteq\sigma_{SBF_+^-}(\ST),
$$
where $B^*\in B(\Y^*)$ denotes the adjoint map of $B\in B(\Y)$ and $\Y^*$ stands for
the dual space of $\Y$.\par

\indent Note that the results concerning the left-right
multiplication operator can be proved using arguments similar to
the ones developed for the tensor product operator, so that, to
avoid repetition, only the proofs regarding property $(gw)$ for
the latter operator will be fully given. \par

\indent First of all, the relationships among the generalized $a$-Weyl spectrum inclusion and both the generalized $a$-Browder's theorem
 and property $(gw)$ will be studied.
To this end, set
$$
\mathbb{S}_a= \sigma_a(A)\sigma_{SBF_+^-}(B)\cup \sigma_{SBF_+^-}(A)\sigma_a(B),\hskip.3truecm
\mathbb{S}'_a= \sigma_a(A)\sigma_{SBF_+^-}(B^*)\cup \sigma_{SBF_+^-}(A)\sigma_a(B^*).
$$

\begin{thm}\label{thm15}Let $\X$ and $\Y$ be two Banach spaces and consider
$A\in B(\X)$ and $B\in B(\Y)$ such that $A$
 satisfies generalized $a$-Browder's theorem.\par
\noindent (a)  If $B\in B(\Y)$ satisfies generalized
$a$-Browder's theorem and the generalized $a$-Weyl spectrum
inclusion for $A\otimes B$ holds, then $A\otimes B\in B(\X
\overline{\otimes} \Y)$ satisfies generalized $a$-Browder's
theorem.\par \noindent (b)  If $B^*\in B(\Y^*)$ satisfies
generalized $a$-Browder's theorem and the generalized $a$-Weyl
spectrum inclusion for $\ST$ holds, then $\ST\in B(B(\Y, \X))$
satisfies generalized $a$-Browder's theorem.

\end{thm}
\begin{proof} (a). Note that according to \cite[Remark 2.7 and Theorem 2.8]{BK},
generalized $a$-Browder's theorem holds for $A\in B(\X)$
(respectively for $B\in B(\Y)$) if and only if acc
$\sigma_a(A)\subseteq \sigma_{SBF_+^-}(A)$ (respectively acc
$\sigma_a(B)\subseteq \sigma_{SBF_+^-}(B)$). As a result, the
generalized $a$-Weyl spectrum inclusion for $A\otimes B$ implies
that
$$
\sigma_a(A)(\hbox{\rm acc }\sigma_a(B))\cup (\hbox{\rm acc }\sigma_a(A))\sigma_a(B)\subseteq \sigma_{SBF_+^-}(A\otimes B).
$$
However, according to  \cite[Theorem 6]{HK} and to the fact that  $\sigma_a(A\otimes B)=\sigma_a(A)\sigma_a(B)$
(\cite[Theorem 4.4]{I}),
$$
\hbox{\rm acc } \sigma_a(A\otimes B)\subseteq \sigma_a(A)(\hbox{\rm acc }\sigma_a(B))\cup (\hbox{\rm acc }
\sigma_a(A))\sigma_a(B)\subseteq \sigma_{SBF_+^-}(A\otimes B).
$$
Therefore,  generalized $a$-Browder's theorem holds for $A\otimes B$.\par
\noindent (b). Adapt the proof of (a) using in particular that $\sigma_a(\ST)=\sigma_a(A)\sigma_a(B^*)$
(\cite[Proposition 4.3(i)]{BDJ}).
\end{proof}

\begin{rema}\label{rem17} \rm Note that the converse implication of Theorem \ref{thm15} does not in general hold.
In fact, if $A\in B( \X)$ is nilpotent, then $A\otimes B\in B(\X
\overline{\otimes}  \Y)$ is nilpotent for every $B\in B(\Y)$. In
addition, according to \cite[Theorem 4]{K}, $\Pi^l(A)=\{0\}=\Pi^l
(A\otimes B)$. As a result,  $A$ and $A\otimes B$ satisfy generalized
$a$-Browder's theorem and
$\sigma_{SBF_+^-}(A)=\emptyset=\sigma_{SBF_+^-}(A\otimes B)$. Let
$B\in B(\Y)$ be such that generalized $a$-Browder's theorem holds
for $B$ and $\sigma_{SBF_+^-}(B)\neq \emptyset$. Then, since
$\mathbb{S}_a=\{0\}$, the $a$-Weyl spectrum inclusion for
$A\otimes B$ does not hold.\par \indent A similar example can be
produced for the left-right multiplication operator.
\end{rema}

\begin{thm}\label{thm18}Let $\X$ and $\Y$ be two Banach spaces and consider
$A\in B(\X)$ and $B\in B(\Y)$ such that $A$ and $B$ are isoloid
operators and $A$ satisfies property $(gw)$. \par \noindent (a)
If $B\in B(\Y)$ satisfies property $(gw)$ and  the generalized
$a$-Weyl spectrum inclusion for $A\otimes B$ holds, then
$A\otimes B\in B(\X  \overline{\otimes}  \Y)$ satisfies property
$(gw)$. \par \noindent (b) If $B^*\in B(\Y^*)$ satisfies property
$(gw)$ and  the generalized $a$-Weyl spectrum inclusion for $\ST$
holds, then $\ST\in B(B(\Y, \X))$ satisfies property $(gw)$.

\end{thm}
\begin{proof} (a). According to \cite[Theorem 2.6]{ABe} and Theorem \ref{thm15},
generalized $a$-Browder's theorem holds for $A\otimes B$.
Furthermore, since $A$ and $B$ are polaroid, see Lemma
\ref{lem16}, $A\otimes B$ is polaroid \cite{DHK}. Observe that
$A$ and $B$ isoloid implies that $A\otimes B$ is isoloid. Hence
$E(A\otimes B)=\Pi (A\otimes B)\subseteq \Pi^l (A\otimes B)$.
Consequently, to conclude the proof it would suffice to prove that
$ \Pi^l (A\otimes B)\subseteq \Pi (A\otimes B)$.\par
\markright{ \hskip5truecm \rm ENRICO BOASSO and B. P. DUGGAL}
\indent Since $\sigma_a(A\otimes B)=\sigma_a(A)\sigma_a (B)$
 and $A$, $B$ and $A\otimes B$ satisfy
generalized $a$-Browder's theorem, according to the generalized
$a$-Weyl spectrum inclusion for $A\otimes B$ and Remark
\ref{rem4},
\begin{align*}
\Pi^l(A\otimes B)&=\sigma_a(A\otimes B)\setminus \sigma_{SBF_+^-}(A\otimes B)\subseteq (\sigma_a(A)\setminus \sigma_{SBF_+^-}(A)) (\sigma_a(B)\setminus \sigma_{SBF_+^-}(B)) \\
&=\Pi^l(A)\Pi^l(B)=\Pi (A)\Pi (B).\\
\end{align*}

\noindent Now well, some cases must be considered. \par

\indent If $0\notin \Pi (A)\Pi (B)$, then, according to \cite[Theorem 6]{HK},
 $\Pi^l(A\otimes B)\subseteq \Pi (A)\Pi (B)=(\Pi (A)\setminus\{ 0\})(\Pi (B)\setminus\{ 0\})=\Pi (A\otimes B)\setminus\{ 0\}\subseteq \Pi (A\otimes B)$.
On the other hand, if $0\in \Pi (A)\cap \Pi (B)$, equivalently, if $A$ and $B$ are Drazin invertible, then
it is not difficult to prove that $A\otimes B$ is Drazin invertible, so that $0\in \Pi (A\otimes B)$.
Consequently, since $\Pi^l(A\otimes B)\setminus \{0\}\subseteq(\Pi (A)\setminus\{ 0\})(\Pi (B)\setminus\{ 0\})=\Pi (A\otimes B)\setminus\{0\}$, 
 $\Pi^l(A\otimes B)\subseteq  \Pi (A\otimes B)$.\par

\indent Next suppose that $0\in \Pi (A)$ and $0\notin\Pi (B)$ (the case
 $0\in \Pi (B)$ and $0\notin\Pi (A)$ can be proved interchanging $A$ with $B$).
If $ \sigma_{SBF_+^-}(B)\neq\emptyset$, then, according to the generalized
$a$-Weyl spectrum inclusion for $A\otimes B$, 
$0\notin \Pi^l(A\otimes B)$. As a result, $ \Pi^l(A\otimes B)\subseteq 
(\Pi (A)\setminus\{ 0\})(\Pi (B)\setminus\{ 0\})=\Pi (A\otimes B)\setminus\{ 0\}\subseteq \Pi (A\otimes B)$.
Finally, suppose that $0\in \Pi (A)$ and $0\notin \Pi(B)=\sigma_a(B)$. Clearly, $0\in\sigma_a(A\otimes B)
\subseteq \sigma (A\otimes B)$. If $0\notin \Pi (A\otimes B)$, then since $A\otimes B$ is polaroid,
$0\in $ acc $\sigma (A\otimes B)$. In particular, $0\in $ acc $\sigma (A)$ or $0\in $ acc $\sigma (B)$.
However, $0\in \Pi (A)$ and since  $ \sigma_{LD}(B)=\sigma_{SBF_+^-}(B)=\emptyset$, according to
\cite[Theorem 2.7]{BBO}, $B$ is algebraic, which implies that $\sigma (B)=\Pi (B)$. Consequently,
$0\in\Pi (A\otimes B)$ and as before $ \Pi^l(A\otimes B)\subseteq \Pi (A\otimes B)$.\par
\noindent (b). Adapt the proof of (a) using in particular that $\sigma_a(\ST )=\sigma_a (A)\sigma_a (B^*)$, 
$\ST\in B(B(\Y ,\X))$ is polaroid (\cite[Lemma 4.7]{BDJ}), $\Pi (B)=\Pi (B^*)$ (\cite[Theorem 2.8]{AS}),
and that $B$ is Drazin invertible (respectively algebraic) if and only if $B^*$ is Drazin invertible (respectively algebraic).
\end{proof}
\markright{ \hskip5truecm \rm Property $(gw)$}
\indent As it will be shown, the transfer property for property
$(gw)$ does not in general imply the generalized $a$-Weyl
spectrum inclusion, both for the tensor product operator and the
left-right multiplication operator. To characterize when this
implication holds, some preparation is needed. \par

\indent In first place, given a Banach space $\X$ and $A\in B(\X)$,
set $I^l(A)=$ (iso $\sigma_a(A)$)$\setminus \Pi^l(A)$. In addition, if $\Y$ is another Banach space and
$B\in B(\Y)$, set
$$
{\mathbb L}_a=(I^l(A)\setminus \{0\})(I^l(B)\setminus \{0\})\cup (I^l(A)\setminus \{0\})(\Pi^l(B)\setminus \{0\})
\cup (\Pi^l(A)\setminus \{0\})(I^l(B)\setminus \{0\}),
$$
$$
{\mathbb L}'_a=(I^l(A)\setminus \{0\})(I^l(B^*)\setminus \{0\})\cup (I^l(A)\setminus \{0\})(\Pi^l(B^*)\setminus \{0\})
\cup (\Pi^l(A)\setminus \{0\})(I^l(B^*)\setminus \{0\}).
$$
\indent Furthermore, note that according to \cite[Theorem
4.4]{I}, \cite[Proposition 4.3(i)]{BDJ} and \cite[Theorem 6]{HK},
$$
\hbox{\rm iso }\sigma_a (A\otimes B)\setminus \{0\}={\mathbb L}_a\cup  (\Pi^l(A)\setminus \{0\})(\Pi^l(B)\setminus \{0\}),
$$
$$
\hbox{\rm iso }\sigma_a (\ST)\setminus \{0\}={\mathbb L}'_a\cup  (\Pi^l(A)\setminus \{0\})(\Pi^l(B^*)\setminus \{0\}).\\
$$

\indent In the following proposition the non-null isolated points
of the approximate point spectrum both of the tensor product
operator and of the left-right multiplication operator will be
described in a particular case.\par

\begin{pro}\label{pro19} Let $\X$ and $\Y$ be two Banach spaces and consider
$A\in B(\X)$ and $B\in B(\Y)$ such that $A$ and $B$ are isoloid operators and $A$
satisfies property $(gw)$. \par
\noindent (a) If $B$ and $A\otimes B\in   B(\X  \overline{\otimes}  \Y)$ satisfy property $(gw)$, then\par
\noindent (i) $\Pi^l(A\otimes B)\setminus \{0\}=(\Pi^l (A)\setminus \{0\})(\Pi^l (B)\setminus \{0\})=(\Pi (A)\setminus \{0\})(\Pi (B)\setminus \{0\}).$\par
\noindent (ii) $ I^l(A\otimes B)\setminus \{0\}={\mathbb L}_a$.\par
\noindent (b) If $B^*$ and $\ST \in   B(B(  \Y,\X))$ satisfy property $(gw)$, then\par
\noindent (iii) $\Pi^l(\ST)\setminus \{0\}=(\Pi^l (A)\setminus \{0\})(\Pi^l (B^*)\setminus \{0\})=(\Pi (A)\setminus \{0\})(\Pi (B^*)\setminus \{0\}).$\par
\noindent (iv) $ I^l(\ST)\setminus \{0\}={\mathbb L}'_a$.\par
\end{pro}
\begin{proof} (a). Note that according to the proof of Theorem \ref{thm18},  $A\otimes B\in B(\X  \overline{\otimes}  \Y)$
is polaroid. Consequently, according to Remark \ref{rem4} and \cite[Theorem 6]{HK},
$$
\Pi^l(A\otimes B)\setminus \{0\}=
\Pi (A\otimes B)\setminus \{0\}=(\Pi (A)\setminus \{0\})(\Pi (B)\setminus \{0\})
=(\Pi^l (A)\setminus \{0\})(\Pi^l (B)\setminus \{0\}).
$$

\indent On the other hand, note that since, according to Lemma
\ref{lem16}, iso $\sigma (A)=\Pi (A)\subseteq \Pi^l(A)$,
$I^l(A)\subseteq $ acc $\sigma(A)$. Similarly, $I^l(B)\subseteq $
acc $\sigma(B)$. As a result, ${\mathbb L}_a\subseteq$ acc
$\sigma (A\otimes B)$. In particular, according to what has been
proved, ${\mathbb L}_a\cap ( \Pi^l(A\otimes B)\setminus
\{0\})=\emptyset$. Hence, since
 iso $\sigma_a (A\otimes B)\setminus \{0\}={\mathbb L}_a\cup (\Pi^l(A\otimes B)\setminus \{0\})$,
statement (ii) holds.\par

\noindent (b). Adapt the proof of (a) using in particular that $\sigma_a(\ST)=\sigma_a(A)\sigma_a(B^*)$, 
$\ST\in B(B(\Y, \X))$ is polaroid 
and the fact that $B\in B(\Y)$ is polaroid if and only if $B^*\in B(\Y^*)$ is polaroid (\cite[Theorem 2.8]{AS}).
\end{proof}

\indent Note further that when
$A\in B(\X)$, $B\in B(\Y)$ and $A\otimes B\in B(\X\overline{\otimes}\Y)$ satisfy property $(gw)$, then
as the following proposition shows, the generalized $a$-Weyl
spectrum inclusion is indeed an equality. However, since for
the main objective of this article the relevant condition
consists in an inclusion, the generalized $a$-Weyl
spectrum inclusion instead of the corresponding equality will be
focused on.\par


 \begin{pro}\label{pro0} Let $\X$ and $\Y$ be two Banach spaces and consider  $A\in B(\X)$ and
$B\in B(\Y)$.\par \noindent (a) If $A$, $B$ and $A\otimes B$
satisfy property $(gw)$, then
$$ \sigma_{{SBF}_+^-}(A\otimes B)\subseteq \sigma_a(A)
\sigma_{{SBF}_+^-}(B)\cup \sigma_{{SBF}_+^-}(A)\sigma_a(B).$$
\noindent  (b) If $A$, $B^*$ and $\ST$ satisfy property $(gw)$, then
$$ \sigma_{{SBF}_+^-}(\ST)\subseteq \sigma_a(A)
\sigma_{{SBF}_+^-}(B^*)\cup \sigma_{{SBF}_+^-}(A)\sigma_a(B^*).$$

\end{pro}

\begin{proof} (a). If $0\notin
(\sigma_a(A)\sigma_{{SBF}_+^-}(B)\cup \sigma_{{SBF}_+^-}(A)\sigma_a(B))$, then either

\noindent (i) $0\notin\sigma_a(A\otimes B)=\sigma_a(A)\sigma_a(B)$, or

\noindent (ii) $0\notin\sigma_a(A)$ and
$0\in(\sigma_a(B)\setminus \sigma_{{SBF}_+^-}(B))$ and
$\sigma_{{SBF}_+^-}(A)=\emptyset$, or
\markright{ \hskip5truecm \rm ENRICO BOASSO and B. P. DUGGAL}
\noindent (iii) $0\notin\sigma_a(B)$ and
$0\in(\sigma_a(A)\setminus \sigma_{{SBF}_+^-}(A))$ and
$\sigma_{{SBF}_+^-}(B)=\emptyset$, or

\noindent (iv) $0\in \sigma_a(A)\cap \sigma_a(B)$
and $\sigma_{{SBF}_+^-}(A)=\sigma_{{SBF}_+^-}(B)=\emptyset$.

\noindent Evidently, $0\notin  \sigma_{{SBF}_+^-}(A\otimes B)$, if
(i) holds.  If (ii) holds, then $\sigma_{{SBF}_+^-}(A)=\emptyset$
implies that $\sigma_a(A)=\Pi^l(A)=E(A)$ and
$0\in(\sigma_a(B)\setminus\sigma_{{SBF}_+^-}(B))$ implies that $0\in\Pi^l(B)=E(B)$ (\cite[Theorem 2.6]{ABe}).
Hence $0\in  E(A\otimes B)=\Pi^l(A\otimes B)=\sigma_a(A\otimes B)\setminus\sigma_{{SBF}_+^-}(A\otimes B)$.
Similarly it is seen that
$0\notin \sigma_{{SBF}_+^-}(A\otimes B)$ in case (iii) or (iv) holds.
Conclusion: $0\notin(\sigma_{{SBF}_+^-}(A)\sigma_a(B) \cup
\sigma_a(A) \sigma_{{SBF}_+^-}(B))$ implies that $0\notin
\sigma_{{SBF}_+^-}(A\otimes B)$. \par

\indent Now let $(0\neq) \lambda\in\sigma_a(A\otimes B)\setminus
(\sigma_{{SBF}_+^-}(A)\sigma_a(B) \cup \sigma_a(A)
\sigma_{{SBF}_+^-}(B))$. Then, for every $\mu\in\sigma_a(A)$ and
$\nu\in\sigma_a(B)$ such that $\lambda= \mu \nu$,
$\mu\notin \sigma_{{SBF}_+^-}(A)$ and
$\nu\notin\sigma_{{SBF}_+^-}(B)$, which implies that
$\mu\in\Pi^l(A)=E(A)$ and $\nu\in\Pi^l(B)=E(B)$. But then using \cite[Theorem 6]{HK},
it is not difficult to prove that $\lambda\in E(A\otimes B)=\Pi^l(A\otimes B)$.
Therefore, according to \cite[Theorem 2.6]{ABe},
$\lambda\notin \sigma_{{SBF}_+^-}(A\otimes B)$.

\noindent (b) Adapt the proof of (a) using in particular
$\sigma_a(\ST)=\sigma_a(A)\sigma_a(B^*)$.\end{proof}

\indent Next the problem of determining when the transfer property for property $(gw)$  implies
 the generalized $a$-Weyl spectrum inclusion  will be studied. In first place,
a particular case will be considered.\par

\begin{thm}\label{thm20} Let $\X$ and $\Y$ be two Banach spaces and consider
$A\in B(\X)$ and $B\in B(\Y)$  two isoloid operators.  Suppose
in addition that $A$ satisfies property $(gw)$.\par \noindent (a)
If  $B$ and $A\otimes B\in   B(\X  \overline{\otimes}  \Y)$
satisfy property $(gw)$, then the following statements hold.
\par
\noindent \rm (i) \it If $I^l(A)=\{0\}$, then the generalized
$a$-Weyl spectrum inclusion for $A\otimes B$ holds if and only if
$\sigma (B)\neq \{0\}$.\par

\noindent \rm (ii)\it If $\sigma_a(A)=\Pi^l(A)=\{0\}$, then the generalized $a$-Weyl's spectrum inclusion for $A\otimes B$
holds if and only if  $\sigma_{SBF_+^-}(B)=\emptyset$.\par

\noindent \rm (iii) \it If $ \{0\} = \Pi^l(A)\subsetneq \sigma_a
(A)$ and $0\notin\sigma_a(B)$, then necessary and sufficient for
the generalized $a$-Weyl spectrum inclusion for $A\otimes B$ to
hold is that $\sigma_{SBF_+^-}(B)=\emptyset$.\par

\noindent \rm (iv) \it If $\Pi^l(A)= \{0\}$ and
$0\in$ iso $\sigma(B)$, then necessary and sufficient for the
generalized $a$-Weyl spectrum inclusion for $A\otimes B$ to hold
is that $\sigma_{SBF_+^-}(A)= \sigma_{SBF_+^-}(B)=\emptyset$.\par

\noindent \rm (v) \it  If $\{0\}=\Pi^l(A)\subsetneq \sigma_a (A)$,
 $0\in$ iso $\sigma_a(B)$ and $\{0\}\notin$ iso $\sigma(B)$, then the generalized $a$-Weyl
spectrum inclusion for $A\otimes B$  holds.\par
\markright{ \hskip5truecm \rm Property $(gw)$}
\noindent \rm (vi) \it If $\{0\}=\Pi^l(A)\subsetneq \sigma_a (A)$,  and $0\in $
acc $\sigma_a (B)$, then the generalized $a$-Weyl spectrum inclusion for
$A\otimes B$  holds.\par

 \noindent (b) If  $B^*$ and
$\ST\in B(B(\Y,\X))$ satisfy property $(gw)$, then the following
statements holds.\par
\noindent \rm (vii) \it If $I^l(A)=\{0\}$,
then the generalized $a$-Weyl spectrum inclusion for $\ST$ holds
if and only if $\sigma (B^*)\neq \{0\}$.\par

\noindent \rm (viii)\it If $\sigma_a(A)=\Pi^l(A)=\{0\}$, then the generalized $a$-Weyl's spectrum inclusion for $\ST$
holds if and only if  $\sigma_{SBF_+^-}(B^*)=\emptyset$.\par

\noindent \rm (ix) \it If
$\{0\}=\Pi^l(A)\subsetneq \sigma_a (A)$ and $0\notin\sigma_a(B^*)$, then necessary and
sufficient for the generalized $a$-Weyl spectrum inclusion for
$\ST$ to hold is that $\sigma_{SBF_+^-}(B^*)=\emptyset$.\par

\noindent \rm (x) \it If $\Pi^l(A)= \{0\}$ and
$0\in$ iso $\sigma(B)$, then necessary and sufficient for the
generalized $a$-Weyl spectrum inclusion for $\ST$ to hold is that
$\sigma_{SBF_+^-}(A)= \sigma_{SBF_+^-}(B^*)=\emptyset$.\par

\noindent \rm (xi) \it  If $\{0\}=\Pi^l(A)\subsetneq \sigma_a
(A)$, $0\in$ iso $\sigma_a(B^*)$ and
$\{0\}\notin$ iso $\sigma(B)$, then the generalized $a$-Weyl
spectrum inclusion for $\ST$  holds.\par

\noindent \rm (xii) \it If $\{0\}=\Pi^l(A)\subsetneq \sigma_a (A)$,  and $0\in $
acc $\sigma_a (B^*)$, then the generalized $a$-Weyl spectrum inclusion for
$\ST$  holds.\par
\end{thm}

\begin{proof}(i). If $\sigma(B)= \{0\}$, then
$\sigma (A\otimes B)=\{0\}$. In addition,
according to Lemma \ref{lem16} and \cite{DHK}, $A\otimes B$ is
polaroid. Consequently, according to Remark \ref{rem4},
$\{0\}=\Pi(A\otimes B)= \Pi^l(A\otimes B)$. As a result,
$\sigma_{{SBF}_+^-}(A\otimes B)=\emptyset$ (\cite[Theorem
2.6]{ABe}). On the other hand, since
$\{0\}=I^l(A)=$ iso $\sigma_a(A)\setminus\Pi^l(A) \subseteq
\sigma_{{SBF}_+^-}(A)$ (\cite[Theorem 2.8]{BK}), $ 0\in\mathbb{
S}_a$ ($= \sigma_a(A)\sigma_{{SBF}_+^-}(B) \cup
\sigma_{{SBF}_+^-}(A)\sigma_a(B)$). Therefore, the generalized
$a$-Weyl spectrum inclusion for $A\otimes B$ does not hold.\par

\indent For the converse,
observe that $0\in \sigma_a(A\otimes B)$ if and only if
$0\in (\sigma_a(A)\cup \sigma_a(B))$,
and in this case $0\in\mathbb{ S}_a$. Now well,
since $A$ is polaroid, $0\in$ acc $\sigma(A)$, and
since $\sigma(B)\neq\{0\}$, clearly $0\in$ acc $\sigma (A\otimes B)$.
However, since $A\otimes B$ is polaroid, according to Remark \ref{rem4},
$0\in\sigma_a(A\otimes B)\setminus\Pi^l(A\otimes B)=\sigma_{{SBF}_+^-}(A\otimes B)$ (\cite[Theorem 2.6]{ABe}).
Now consider $\mathbb{S}_a\setminus \{0\}$. Since $A$ and $B$ satisfy the generalized $a$-Browder's theorem,
according to \cite[Theorem 2.8]{BK},
 $\sigma_{SBF_+^-}(A)\setminus  \{0\}=$ acc $\sigma_a(A)$ and $\sigma_{SBF_+^-}(B)\setminus \{0\}
=$ (acc $\sigma_a(B)\setminus \{0\})\cup ( I^l(B)\setminus \{0\})$. As a result, a straightforward calculation, using in particular Proposition \ref{pro19}(ii),
proves that $\mathbb{S}_a\setminus \{0\}\subseteq $ acc $\sigma_a (A\otimes B)\cup \mathbb{L}_a$.
Consequently, $(\mathbb{S}_a\setminus \{0\})\cap (\sigma_a(A\otimes B)\setminus \sigma_{SBF_+^-}(A\otimes B))
= (\mathbb{S}_a\setminus \{0\})\cap \Pi^l(A\otimes B)=\emptyset$, equivalently, $\mathbb{S}_a\setminus \{0\}\subseteq \sigma_{SBF_+^-}(A\otimes B)$.\par

\noindent (ii).  Suppose that $\sigma_a (A)=\Pi^l(A)=\{0\}$. Then, $\sigma_{SBF_+^-}(A)=\sigma_{LD}(A)=\emptyset$
and, according to \cite[Theorem 2.7]{BBO}, $A$ is algebraic.
In addition, according to \cite[Theorem 1.5]{BKMO} and \cite[Theorems 3 and 12]{Bo},
$\sigma (A)=\Pi(A)=\{0\}$. As a result, $A$ is nilpotent, which
implies that $A\otimes B$ is nilpotent. Hence, according to \cite[Theorem 4]{K}, $\Pi(A\otimes B)=\Pi^l(A\otimes B)=\{0\}$
and $\sigma_{SBF_+^-}(A\otimes B)=\emptyset$. However,
${\mathbb S}_a =\sigma_a(A)\sigma_{SBF_+^-}(B)$. Therefore, ${\mathbb S}_a=\emptyset$
if and only if $\sigma_{SBF_+^-}(B)=\emptyset$.\par

\noindent (iii). The hypotheses imply that
$\sigma_a(A)\setminus\{0\}=\sigma_{{SBF}_+^-}(A)\neq \emptyset$. Assume that
$\sigma_{{SBF}_+^-}(B)=\emptyset$. Hence, if $0\notin \sigma_a(B)$, then
$\mathbb{S}_a=(\sigma_a(A)\setminus\{0\}) \sigma_a(B)=
\sigma_a(A\otimes B)\setminus\{0\}$. Observe that
$\Pi^l(A)=\Pi(A)=\{0\}$, and so (since $0\notin\sigma_a(B)$ and $A\otimes B$ is polaroid)
$0\in\Pi(A\otimes B)=\Pi^l(A\otimes B)$. However, since $A$, $B$ and $A\otimes B$ are polaroid operators
and $\Pi^l(A)=\{0\}$, according
to \cite[Theorem 6]{HK}, $\Pi^l(A\otimes B)= \Pi(A\otimes B)=\{0\}$.
Therefore, $\sigma_{{SBF}_+^-}(A\otimes B)=\sigma_a(A\otimes B)\setminus \Pi^l(A\otimes B)=
\sigma_a(A\otimes B)\setminus  \{0\}= \mathbb{S}_a$.\par

\indent Conversely, if the generalized $a$-Weyl spectrum inclusion holds
and $\sigma_{{SBF}_+^-}(B)\neq \emptyset$, then $0\in
\sigma_a(A)\sigma_{{SBF}_+^-}(B)$, which implies that $0\in
\sigma_{{SBF}_+^-}(A\otimes B)$, and  which in turn implies that
$0\notin\Pi^l(A\otimes B)= \Pi(A\otimes B)$. However,
this is a contradiction, for $0\in \Pi^l(A\otimes B)= \Pi(A\otimes B)$.\par
\markright{ \hskip5truecm \rm ENRICO BOASSO and B. P. DUGGAL}
\noindent (iv). Suppose that  $\Pi^l(A)=\{0\}$ and
$0\in$  iso $\sigma(B)$. Then, since $B$ isoloid implies that $0\in
E(B)$ and  $A$ and $A\otimes B$ satisfy property $(gw)$,
 it follows from \cite[Theorem 2.6]{ABe}
that $\Pi^l(A)=E(A)$ and $0\in E(A\otimes
B)=\Pi^l(A\otimes B)$. In addition, since $A$, $B$, and $A\otimes
B$ are polaroid, according to \cite[Theorem 6]{HK},
$\Pi^l(A\otimes B)=\{0\}$. In particular, $0\notin
\sigma_{{SBF}_+^-}(A\otimes B)$. However, since $0\in
\sigma_a(A)\cap\sigma_a(B)$ ($E(B)=\Pi^l(B)$), the generalized $a$-Weyl spectrum
inclusion holds for $A\otimes B$ if and only if
 $\sigma_{{SBF}_+^-}(A)=\sigma_{{SBF}_+^-} B)=\emptyset$.\par
\markright{ \hskip5truecm \rm Property $(gw)$}
\noindent (v). The hypotheses $0\in$ iso $\sigma_a(B)$ and
$\{0\}\notin$ iso $\sigma(B)$ imply that $0\in I^l(B)$. Since
$\sigma_{{SBF}_+^-}(A)=\sigma_a(A)\setminus \{0\}$ and  $0\in
\sigma_a(A)$, a straightforward calculation proves that
$\mathbb{S}_a=\sigma_a(A\otimes B)$. On the other hand, since
$0\in \Pi^l(A)\cap I^l(B)$, $0\in$ iso $\sigma_a(A\otimes B)$.
Moreover, since $B$ is polaroid, $0\in $ acc $\sigma(B)$, and
since $\sigma_a(A)\setminus \{0\}\neq \emptyset$, $0\in $ acc
$\sigma(A\otimes B)$. As a result, $0\notin\Pi(A\otimes
B)=\Pi^l(A\otimes B)$. However, according to \cite[Theorem
6]{HK}, $\Pi^l(A\otimes B)=\Pi(A\otimes B) =\emptyset$.
Consequently, $\sigma_{{SBF}_+^-}(A\otimes B)=\sigma_a(A\otimes
B)$.\par

\noindent (vi). If $0\in $ acc $\sigma_a(B)$ and $\sigma_a(A)\neq \{0\}$,
then a straightforward calculation proves that $0\in$ acc $\sigma_a(A\otimes B)$.
However, since $A\otimes B$ is polaroid and $\Pi (A)=\Pi^l(A)=\{0\}$,
according to \cite[Theorem 6]{HK}, $\Pi^l(A\otimes B)=\Pi(A\otimes B)=\emptyset$.
Therefore, $\sigma_a(A\otimes B)=\sigma_{{SBF}_+^-}( A\otimes B)$ and the
generalized $a$-Weyl spectum inclusion trivially holds.\par

\noindent (vii)-(xii).  Adapt the proof of (i)-(vi) using in particular that $\sigma_a(\ST)=\sigma_a(A)\sigma_a(B^*)$, 
that $\ST$ is polaroid 
and the fact that $B\in B(\Y)$ is polaroid if and only if $B^*\in B(\Y^*)$ is polaroid.
\end{proof}

\indent Note that under the same conditions of Theorem \ref{thm20},
if instead of $I^l(A)= \{0\}$ or $\Pi^l(A)=\{0\}$ the conditions
$I^l(B)= \{0\}$ or $\Pi^l(B)=\{0\}$
are assumed, then results similar to the ones in  Theorem \ref{thm20}
can be proved.\par
\markright{ \hskip5truecm \rm ENRICO BOASSO and B. P. DUGGAL}
\indent In what follows the study of the relationship
between property $(gw)$ and the generalized $a$-Weyl spectrum
inclusion both for the tensor product operator and the left-right
multiplication operator will be completed. To this end,
 given $\X$ and $\Y$ two
Banach spaces and $A\in B(\X)$ and $B\in B(\Y)$, set
$$
{\mathbb A}_a= \sigma_a (A) (\hbox{\rm acc } \sigma_a(B))\cup (\hbox{\rm acc } \sigma_a(A))\sigma_a(B),\hskip.2truecm
{\mathbb B}_a=I^l(A)I^l(B)\cup I^l(A)\Pi^l(B)\cup \Pi^l(A)I^l(B),
$$
$$
{\mathbb A}'_a= \sigma_a (A) (\hbox{\rm acc } \sigma_a(B^*))\cup (\hbox{\rm acc } \sigma_a(A))\sigma_a(B^*),
\hskip.1truecm{\mathbb B}_a=I^l(A)I^l(B^*)\cup I^l(A)\Pi^l(B^*)\cup \Pi^l(A)I^l(B^*).
$$

Observe that if property $(gw)$ holds for $A$ and
$B$ (respectively for $A$ and $B^*$), then ${\mathbb
S}_a={\mathbb A}_a\cup {\mathbb B}_a$ (respectively ${\mathbb
S}'_a={\mathbb A}'_a\cup {\mathbb B}'_a$). In fact, in the tensor
product operator case and under the above mentioned conditions,
$\sigma_{SBF_+^-}(A) =\sigma_{LD}(A)=$ acc $\sigma_a(A)\cup
I^l(A)$ and $\sigma_{SBF_+^-}(B) =\sigma_{LD}(B)=$ acc
$\sigma_a(B)\cup I^l(B)$. A similar argument proves the left-right
multiplication operator case.
\par

\indent Furthermore, according to Proposition \ref{pro0}, if $\sigma_{SBF_+^-}(A) =\sigma_{SBF_+^-}(B)
=\emptyset$, (respectively if $\sigma_{SBF_+^-}(A) =\sigma_{SBF_+^-}(B^*)=\emptyset$)
then $\sigma_{SBF_+^-}(A\otimes B)=\emptyset$ (respectively $\sigma_{SBF_+^-}(\ST)=\emptyset$). In particular, the generalized $a$-Weyl spectrum inclusion for $A\otimes B$
(respectively for $\ST$)
holds. As a result, to conclude this work, two cases need to be considered, namely, (i) when only one of the upper B-Weyl spectra  of $A$
and $B$ (respectively $B^*$) is the empty set, (ii) when the upper B-Wely spectrum both  of $A$ and of $B$ (respectively of $B^*$) is not empty.
In the next theorem, the first case will be considered.\par

\begin{thm}\label{thm22} Let $\X$ and $\Y$ be two Banach spaces and consider
$A\in B(\X)$ and $B\in B(\Y)$ two isoloid operators. Suppose in
addition that $A$ satisfies property $(gw)$ and $\sigma_a(A)=\Pi (A)\neq\{0\}$.\par

 \noindent (a) If  $B\in B(\Y)$ is
such that $\sigma_{SBF_+^-}(B) \neq\emptyset$ and $B$ and $A\otimes B\in B(\X
\overline{\otimes} \Y)$ satisfy property $(gw)$, then the
following statements are equivalents.\par 

\noindent\rm (i)\it The generalized $a$-Weyl spectrum inclusion for $A\otimes B$
holds.\par

 \noindent \rm (ii) Either $0\notin \Pi(A)$, or, if $0\in\Pi (A)$, then $B$ is not left Drazin invertible.\par
\noindent Furthermore, if  statements (i)-(ii) hold, then
${\mathbb S}_a=\sigma_{SBF_+^-}(A\otimes B)$, while if one of
these statements does not hold, then ${\mathbb S}_a=
\sigma_{SBF_+^-}(A\otimes B)\cup  \{0\}$, $0\notin \sigma_{SBF_+^-}(A\otimes B)$.\par
 \noindent (b) If  $B^*\in B(\Y^*)$ is
such that $\sigma_{SBF_+^-}(B^*) \neq\emptyset$ and $B^*$ and $\ST\in B(B(\Y,\X))$ satisfy property $(gw)$, then the
following statements are equivalents.\par 

\noindent\rm (iii)\it The generalized $a$-Weyl spectrum inclusion for $\ST$
holds.\par

 \noindent \rm (iv) Either $0\notin \Pi(A)$, or, if $0\in\Pi (A)$, then $B^*$ is not left Drazin invertible.\par
\noindent Furthermore, if  statements (i)-(ii) hold, then
${\mathbb S}'_a=\sigma_{SBF_+^-}(\ST)$, while if one of
these statements does not hold, then ${\mathbb S}'_a=
\sigma_{SBF_+^-}(\ST)\cup  \{0\}$, $0\notin \sigma_{SBF_+^-}(\ST)$.\par
\end{thm}
\markright{ \hskip5truecm \rm Property $(gw)$}
\begin{proof} (a). First  of all note that 
$$\sigma_a(A\otimes B)=\Pi(A)\Pi (B)\cup \Pi(A)\sigma_{SBF_+^-}(B),\hskip.3truecm
{\mathbb S}_a=\Pi (A)\sigma_{SBF_+^-}(B)\neq\emptyset.$$ 
\indent Moreover, according to \cite[Theorem 6]{HK} and Proposition \ref{pro19}(ii),
acc $\sigma_a(A\otimes B)\setminus\{0\}= (\Pi(A)\setminus\{0\}) ($acc $\sigma_a(B)\setminus\{0\})$
and $I^l(A\otimes B)\setminus\{0\}=   (\Pi(A)\setminus\{0\})(I^l(B)\setminus\{0\})$. In particular,
$$
\sigma_{SBF_+^-}(A\otimes B)\setminus\{0\}=\hbox{ acc } \sigma_a(A\otimes B)\setminus\{0\}\cup
I^l(A\otimes B)\setminus\{0\}={\mathbb S}_a\setminus\{0\}.
$$
As a result, according to Proposition \ref{pro0}, statement (i) is equivalent to the 
following implication: if $0\in {\mathbb S}_a$, then $0\in \sigma_{SBF_+^-}(A\otimes B)$.\par

\indent  Note that if $0\in \sigma_{SBF_+^-}(B)$, then either $0\in$ acc $\sigma_a(B)$ or $0\in I^l(B)$.
If $0\in$ acc $\sigma_a(B)$, then since $\sigma_a(A)\neq\{0\}$,  $0\in $ acc $\sigma_a(A\otimes B)\subseteq
\sigma_{SBF_+^-}(A\otimes B)$. On the other hand, if $0\in I^l(B)\subseteq $ acc $\sigma (B)$ (Lemma \ref{lem16}),
then since $\sigma_a(A)\neq\{0\}$, $0\in $ acc $\sigma (A\otimes B)$. Thus, $0\notin \Pi(A\otimes B)=\Pi^l(A\otimes B)$,
which implies that $0\in  \sigma_{SBF_+^-}(A\otimes B)$. Consequently, if $0\in \sigma_{SBF_+^-}(B)$, then
clearly the aforementioned implication and  statement (i) hold.\par
\indent Observe that if $0\notin\Pi (A)$ and $0\notin  \sigma_{SBF_+^-}(B)$, then there is nothing to verify, while if
$0\notin\Pi (A)$ and $0\in  \sigma_{SBF_+^-}(B)$, according to what has been proved, $0\in \sigma_{SBF_+^-}(A\otimes B)$.
Now assume that  $0\in \Pi(A)$, equivalently $A$ is Drazin invertible. Naturally, if $0\in \sigma_{SBF_+^-}(B)$, according to what has been proved, the aforementioned
implication and statement (i) hold. However, if $0\notin \sigma_{SBF_+^-}(B)=\sigma_{LD}(B)$, equivalently if $B$ is left Drazin invertible,
it is not difficult to prove that $0\in\Pi^l(A\otimes B)=\Pi(A\otimes B)$. Therefore, only in this case $0\in{\mathbb S}_a\setminus \sigma_{SBF_+^-}(A\otimes B)$.\par
\indent The last statement can be derived from what has been proved.\par
\noindent (b). Adapt the proof of (a) to the case under consideration.
\end{proof}

\indent Naturally, interchanging $A$ with $B$ in Theorem \ref{thm22}, similar results can be proved. Next the last case will be considered.\par

\begin{thm}\label{thm21} Let $\X$ and $\Y$ be two Banach spaces and consider
$A\in B(\X)$ and $B\in B(\Y)$ two isoloid operators. Suppose in
addition that $A$ satisfies property $(gw)$ and $\sigma_{SBF_+^-}(A)\neq\emptyset$.\par

 \noindent (a) If  $B\in B(\Y)$ is
such that $\sigma_{SBF_+^-}(B) \neq\emptyset$ and $B$ and $A\otimes B\in B(\X
\overline{\otimes} \Y)$ satisfy property $(gw)$, then the
following statements are equivalents.\par 

\noindent \rm (i) \it
The generalized $a$-Weyl spectrum inclusion for $A\otimes B$
holds.\par

 \noindent \rm (ii) \it $0\notin\Pi(A\otimes B)=\Pi^l(A\otimes B)$.\par

\noindent \rm (iii)\it Either $A\otimes B$ is bounded below or $0\in \sigma_{SBF_+^-}(A\otimes B)$.\par
\noindent Furthermore, if  statements (i)-(iii) hold, then
${\mathbb S}_a=\sigma_{SBF_+^-}(A\otimes B)$, while if one of
these statements does not hold, then ${\mathbb S}_a=
\sigma_{SBF_+^-}(A\otimes B)\cup  \{0\}$, $0\notin \sigma_{SBF_+^-}(A\otimes B)$.\par

 \noindent (b) If
$B^*\in B(\Y^*)$ is such that $\sigma_{SBF_+^-}(B^*) \neq\emptyset$ and $B^*$
and $\ST \in B(B( \Y,\X))$ satisfy property $(gw)$, then the
following statements are equivalents.\par 

\noindent \rm (iv) \it
The generalized $a$-Weyl spectrum inclusion for $\ST$ holds.\par

\noindent \rm (v) \it $0\notin\Pi(\ST)=\Pi^l(\ST)$.\par 
\markright{ \hskip5truecm \rm ENRICO BOASSO and B. P. DUGGAL}
\noindent \rm (vi)\it Either $\ST$ is bounded below or $0\in \sigma_{SBF_+^-}(\ST)$.\par
\noindent
Furthermore, if  statements (iv)-(vi) hold, then ${\mathbb
S}'_a=\sigma_{SBF_+^-}(\ST)$, while if one of these statements
does not hold, then ${\mathbb S}'_a= \sigma_{SBF_+^-}(\ST)\cup
\{0\}$, $0\notin \sigma_{SBF_+^-}(\ST)$.
\end{thm}
\begin{proof} (a). (i)$\Rightarrow$ (ii). Suppose that $0\in\Pi (A\otimes B)=\Pi^l(A\otimes B)
\subseteq\sigma_a(A\otimes B)$. In particular, $0\in\sigma_a(A)\cup\sigma_a(B)$. However,
since $\sigma_{SBF_+^-}(A)\neq\emptyset$ and  $\sigma_{SBF_+^-}(B)\neq\emptyset$,
$0\in {\mathbb S}_a\setminus\sigma_{SBF_+^-}(A\otimes B)$, which is impossible, for
 ${\mathbb S}_a\subseteq \sigma_{SBF_+^-}(A\otimes B)$.\par 

\noindent (ii)$\Rightarrow$ (iii). Clear.\par

\noindent (iii)$\Rightarrow$ (i). Recall that according to \cite[Theorem 6]{HK}, acc $\sigma_a(A\otimes B)\setminus\{0\}={\mathbb A}_a\setminus\{0\}$.
In addition, since according to Proposition \ref{pro19}(ii), $I^l(A\otimes B)\setminus\{0\}={\mathbb B}_a\setminus\{0\}$,
 $$\sigma_{SBF_+^-}(A\otimes B)\setminus\{0\}= \hbox{ acc }\sigma_a(A\otimes B)\setminus\{0\}\cup I^l(A\otimes B)\setminus\{0\}=  {\mathbb S}_a\setminus\{0\}.$$

\noindent Now well,  if $A\otimes B$ is bounded below, then   $\sigma_{SBF_+^-}(A\otimes B)= {\mathbb S}_a$. On the other hand,
if $0\in  \sigma_{SBF_+^-}(A\otimes B)\subseteq \sigma_a(A\otimes B)$, then $0\in\sigma_a(A)$ or $0\in \sigma_a(B)$. Hence,
since   $\sigma_{SBF_+^-}(A)\neq\emptyset$ and  $\sigma_{SBF_+^-}(B)\neq\emptyset$,
$0\in {\mathbb S}_a$ and then  $\sigma_{SBF_+^-}(A\otimes B)= {\mathbb S}_a$.\par
\indent The last statement can be derived from what has been proved.\par
\noindent (b). Adapt the proof of (a) using ${\mathbb S}'_a$, ${\mathbb A}'_a$ and ${\mathbb B}'_a$ instead of ${\mathbb S}_a$, ${\mathbb A}_a$
and ${\mathbb B}_a$ respectively.
\end{proof}


\markright{ \hskip5truecm \rm Property $(gw)$}

\bigskip
\markright{ \hskip5truecm \rm ENRICO BOASSO and B. P. DUGGAL}

\noindent \normalsize \rm Enrico Boasso\par
  \noindent  E-mail: enrico\_odisseo@yahoo.it \par
\medskip
\noindent B. P. Duggal\par
\noindent E-mail:  bpduggal@yahoo.co.uk


\begin{thebibliography}{20}

\bibitem{A} P. AIENA, Property (w) and perturbations II, {\em J. Math. Anal. Appl.}, {\bf 342} (2008),
830-837.

\bibitem{AB} P. AIENA and M. T. BIONDI, Property (w) and perturbations,
{\em J. Math. Anal. Appl.}, {\bf 336}  (2007), 683-692.

\bibitem{ABV} P. AIENA, M. T. BIONDI and F. VILLAFA\~ NE, Property (w) and perturbations III,
 {\em J. Math. Anal. Appl.}, {\bf 353}  (2009), 205-214.

\bibitem{AGP} P. AIENA, J. R. GUILLEN and P. PE\~ NA, Property (w) for perturbations of polaroid operators,
{\em Linear Algebra Appl.}, {\bf 428} (2008), 1791-1802.

\bibitem{AP} P. AIENA and P. PE\~ NA, Variations on Weyl's theorem,
{\em J. Math. Anal. Appl.}, {\bf  324} (2006), 566-579.

\bibitem{AS} P. AIENA and J. E. SANABRIA, On left and right poles of the resolvent, {\em Acta Sci. Math. (Szeged)}, {\bf 74} (2008), 669-687.

\bibitem{Am} M. AMOUCH,  Polaroid Operators with SVEP and Perturbations of  Property $(gw)$,
{\em Mediterr. J. Math.}, {\bf  6} (2009), 461-470.

\bibitem{ABe} M. AMOUCH and M. BERKANI,  On the Property $(gw)$,
{\em Mediterr. J. Math.}, {\bf 5} (2008), 371-378.
\markright{ \hskip5truecm \rm Property $(gw)$}

\bibitem{AZ} M. AMOUCH and H. ZGUITTI, On the equivalence of Browder's and generalized Browder's theorem,
{\em Glasg. Math. J.}, {\bf 48} (2006), 179-185.

\bibitem{BBO} O. BEL HADJ FREDJ, M. BURGOS and M. OUDGHIRI,
Ascent spectrum and essential ascent spectrum,
{\em Studia Math.}, {\bf 187} (2008), 59-73.
\markright{ \hskip5truecm \rm ENRICO BOASSO and B. P. DUGGAL}
\bibitem{B1} M. BERKANI, On a class of quasi-Fredholm operators, {\em Integral Equations
Operator Theory}, {\bf 34} (1999), 244-249.

\bibitem{B2} M. BERKANI, On semi B-Fredholm operators, {\em Glasg. Math. J.},
{\bf 43} (2001), 457-465.

\bibitem{B3} M. BERKANI, B-Weyl spectrum and poles of the resolvent, {\em J. Math.
Anal. Appl.}, {\bf 272} (2002), 596-603.

\bibitem{BA} M. BERKANI and M. AMOUCH, Preservation of propety  $(gw)$ under perturbations,
{\em Acta Sci. Math. (Szeged)}, {\bf 74} (2008), 769-781.

\bibitem{BK} M. BERKANI and J. J. KOLIHA,   Weyl type theorems for bounded linear operators,
{\em Acta Sci. Math (Szeged)}, {\bf 69} (2003), 359-376.

\bibitem{BS} M. BERKANI and M. SARIH, An Atkinson-type theorem for B-fredholm operators,
{\em Studia Math.}, {\bf 148} (2001), 251-257.

\bibitem{BS2} M. BERKANI and M. SARIH, On semi B-Fredholm operators, {\em Glasg. Math. J.}, {\bf 43} (2001), 457-465.

\bibitem{Bo} E. BOASSO, Drazin spectra of Banach space operators and Banach
algebra elements, {\em J. Math. Anal. Appl.}, {\bf 359} (2009), 48-55.

\bibitem{BDJ} E. BOASSO, B. P. DUGGAL and I. H. JEON, Generalized Browder's and
Weyl's theorems for left and right multiplication operators, {\em J. Math. Anal.
Appl.}, {\bf 370} (2010), 461-471.

\bibitem{BKMO} M. BURGOS, A. KAIDI, M. MBEKHTA and M. OUDGHIRI,
The descent spectrum and perturbations,
{\em J. Operator Theory}, {\bf 56} (2006), 259-271.

\bibitem{CPY} S. R. CARADUS, W. E. PFAFFENBERG and B. YOOD, {\em Calkin algebras and algebra
of operators on Banach spaces}, Marcel Dekker, Inc., New York, 1974.

\bibitem{C} L. A. COBURN, Weyl's theorem for nonnormal operators, {\em Michigan Math. J.}, {\bf 13} (1966),
285-288.


\bibitem{D} B. P. DUGGAL, Tensor products and property $(w)$, {\em Rend. Circ. Mat.
Palermo}, {\bf 60} (2011), 23-30.

\bibitem{DDK} B. P. DUGGAL, S. V. DEJORDJEVI\' C and C. S. KUBRUSLY,
On the $a$-Browder and $a$-Weyl spectra of tensor products,
{\em Rend. Circ. Math. Palermo}, {\bf 59} (2010), 473-481.

\bibitem{DHK} B. P. DUGGAL, R. HARTE and A. H. KIM, Weyl's theorem,
tensor products and multiplication operators II, {\em Glasg. Math. J.},
{\bf 52} (2010), 705-709.

\bibitem{HK} R. E. HARTE and A.  H. KIM,  Weyl's theorem,
tensor products and multiplication operators,  {\em J. Math. Anal.
Appl.}, {\bf 336} (2007), 1124-1131.

\bibitem{I} T. ICHINOSE, Spectral properties of linear operators II,
{\em Trans. Amer. Math. Soc.}, {\bf 237} (1978), 223-254.

\markright{ \hskip5truecm \rm ENRICO BOASSO and B. P. DUGGAL}

\bibitem{K} C. KING,  A note on Drazin inverses, {\em Pacific J. Math.}, {\bf 70} (1977), 383-390.


\bibitem{R1} V. RAKO\v CEVI\'C, On one subset of M. Schechter's essential spectrum,
 {\em Mat. Vesnik}, {\bf 5(18)(33)} (1981), 389-391.

\bibitem{R2} V. RAKO\v CEVI\'C, On the essential approximate point spectrum II,
 {\em Mat. Vesnik}, {\bf 36} (1984), 89-97.

\bibitem{R} V. RAKO\v CEVI\'C, On a class of operators, {\em Mat. Vesnik} {\bf 37} (1985), 423-426.

\bibitem{R3} V. RAKO\v CEVI\'C, Approximate point spectrum and commuting compact pertubations, {\em Glasg.
Math. J.}, {\bf 28} (1986), 193-198.

\end{thebibliography}
\end{document}